\documentclass[12pt]{amsart}
\usepackage{times,amssymb, amsthm}
\usepackage{enumerate,comment}
\usepackage{tikz}
\usepackage{cite}
\usetikzlibrary{positioning}
\usepackage[T1]{fontenc}
\usepackage[utf8]{inputenc}

\parskip=\smallskipamount

\newtheorem*{Thm*}{Theorem}
\newtheorem{Thm}{Theorem}

\newtheorem{Lemma}[Thm]{Lemma}

\theoremstyle{definition}

\newtheorem{Remark}[Thm]{Remark}

\title{Existence of Solutions of the third term of the Connaughton-Newell Model with a source term}
\author{Anh Vi\d{\^{e}}t Nguy\~{\^e}n}
\email{navgoingcollege2023@tamu.edu}

\begin{document}
\maketitle
\begin{abstract}
The Connaughton-Newell equation is an approximation of three-wave kinetic equations using a fully non-linear coagulation-fragmentation model. This equation consists of three non-linear operators. In this paper, we proved that assuming a constant interaction kernel and a well-behaved source term, the third operator of the Connaughton-Newell equation has a solution.
\end{abstract}
\begin{section} {Introduction}
The theory of wave turbulence provides a fundamental framework for understanding a wide range of physical and mechanical phenomena. Some notable examples are waves due to rotation, Alfv\'en wave turbulence in solar wind, and waves in plasma fusion devices. the theory of wave turbulence, first formulated by Peiers\'\ \cite{1},  was further developed by Benney \& Saffman \cite{2}, Zakharov \& Filonenko \cite{3}, Benney \& Newell \cite{4}, and most notably, Hasselmann \cite{5,6}. These works led to the formulation of the three- and four-wave kinetic equations, which describe the energy distribution among weakly interacting waves. The three-wave kinetic equation is written

\[
\partial_tf(t,p):= \int \int_{\mathbb{R}^{2d}}[\mathrm{R}_{p,p_1,p_2}[f] - \mathrm{R}_{p_1,p,p_2}[f] - \mathrm{R}_{p_2,p,p_1}[f]] d p_1 dp_2, \quad f(0,p)=f_0(p),
\]

where, $\mathrm{R}_{p,p_1,p_2}[f]$ is defined as

\[
\mathrm{R}_{p,p_1,p_2}[f] := |V_{p,p_1,p_2}|^2 \delta (p-p_1-p_2) \delta (\omega (p)- \omega(p_1) - \omega(p_2))(f(p_1)f(p_2)-ff(p_1)-ff(p_2)).
\]

\vspace{0.5cm}

 $f(t,p)$ represents the wave density at wave number $p\in \mathbb{R}^d$ with initial condition $f_0(p)$, and $\omega(p)$ represents the dispersion relation of the waves. For further understanding, we direct readers to the comprehensive work in \cite{7,8,9}.\\

In a closely related setting, the coagulation and fragmentation kinetic equations also appear in a variety of physical and mechanical applications. These include: the formation of aerosols, polymers, celestial bodies on an astronomical scale, and colloidal aggregates\cite{10,11,12,13,14,15,16,17,18,19,20,21}. These equations describe the coagulating or fragmentation behavior of a system of particles under some interaction or external force. For the pure coagulation model, the Smoluchowski equation, where Well-Posedness was shown by Shirvani and Roessel \cite{22}, serves as a classical model for particulate processes. In the study of these equations, the main focus is on the particle density function, usually denoted as $n_\omega(t)$ where $\omega \in \mathbb{N}$ denotes the size of particles at time $t\in [0,+\infty)$.

Surprisingly, there is a conceptual similarity between the three-wave kinetic equations and the coagulation-fragmentation kinetic equations. In particular, the transfer of energy between scales in wave turbulence is similar to the transfer of mass between clusters in the coagulation process. In 2010, Connaughtin and Newell proposed an approximation of the three-wave kinetic equations using a non-linear coagulation-fragmentation model \cite{23}. A numerical scheme for this model was provided by Das and Tran \cite{24}, and its analysis was provided by Wang and Tran \cite{25}. The Connaughton-Newell equation expresses both the coagulation and fragmentation terms in a nonlinear manner. The equation read\\
\[
\frac{\partial N_\omega}{\partial t}:=\mathcal{Q}[N_\omega](t),\ \omega\in \mathbb{R^+}, \  N_\omega(0) = N_\omega^{in}
,\]
\\
where the operator$\mathcal{Q}[N_\omega](t)$ is written as

\[
\mathcal{Q}[N_\omega](t) := S_1[N_\omega] + S_2[N_\omega] + S_3[N_\omega],
\]

and

\[
    S_1(N_\omega):=\int_0^\omega K_1(\omega-\mu, \mu) N_\mu N_{\omega-\mu}d\mu -  \int_\omega^\infty K_1(\mu-\omega, \omega) N_\omega N_{\mu-\omega} d\mu- \int_0^\infty K_1(\omega, \mu) N_\omega N_\mu d\mu,
\]

\[
    S_2(N_\omega):=-\int_0^\omega K_2(\mu, \omega-\mu) N_\omega N_\mu d\mu +  \int_\omega^\infty K_2(\omega, \mu-\omega) N_\mu N_{\mu-\omega} d\mu+ \int_0^\infty K_2(\omega, \mu) N_{\omega+\mu} N_\mu d\mu,
\]

\[
    S_3(N_\omega):=-\int_0^\omega K_3(\mu, \omega-\mu) N_\omega N_{\omega-\mu}d\mu +  \int_\omega^\infty K_3(\mu, \mu-\omega) N_\omega N_\mu d\mu+ \int_0^\infty K_3(\mu, \omega) N_{\omega+\mu} N_\mu d\mu.
\]

Here, $N_\mu$ represents the wave frequency spectrum such that $\int_{\mu_1}^{\mu_2}N_\mu d\mu$ is the total wave action in the frequency band $[\mu_1, \mu_2]$. The homogeneous function $K_i(\omega, \mu)$ represents the wave interaction kernels where $K_1$ is associated with the transfer of energy as $K_2$ and $K_3$ are associated with the back-scattering of energy.\\

In this paper, we show that there is a solution to the equation $\frac{\partial N_\omega}{\partial t} = S_3(N_\omega)$ with some additional assumptions. In particular, we look at $S_3(N_\omega)$ with a non-trivial source term, denoted $g_\omega(t)$, and a constant interaction kernel, denoted $K_3(\omega, \mu)$, which will be regarded as 1. We consider the equation\\
\begin{equation} \label{eq1}
\begin{split}
    \frac{\partial N_\omega}{\partial t} &=-\int_0^\omega N_\omega N_{\omega-\mu}d\mu + \int_\omega^\infty N_\omega N_\mu d\mu+ \int_0^\infty N_{\omega+\mu} N_\mu d\mu +g_\omega(t).\\
\end{split}
\end{equation}\\
Equation ~\eqref{eq1} can be rewritten as\\
\begin{equation}
\label{eq2}
\begin{split}
   \frac{\partial N_\omega}{\partial t}= N_\omega \int_0^\infty N_\mu d\mu+ \int_0^\infty N_{\omega+\mu} N_\mu d\mu+g_\omega(t).\\
\end{split}
\end{equation}

To account for physical constraints, we impose the following assumptions:

\begin{itemize}
\item[(A1)] For each $\omega \ge 0$, the source term
\[
g_\omega : [0,T] \to \mathbb{R}
\]
belongs to $L^1_{\mathrm{loc}}[0,T]$.\\

\item[(A2)] The initial value for $N_\omega$ and the source term, $g_\omega$, are nonnegative; that is
\[
N_\omega(0) \ge 0, \ \text{ for all }  \omega \ge 0 \quad \text{and}
\quad
g_\omega(t) \ge 0, \ \text{ for all }  (t,\omega) \in [0,T]\times(0,\infty).\\
\]

\item[(A3)] The initial value is integrable with respect to $\omega$
\[
\int_0^\infty N_\omega(0)\, d\omega < \infty.\\
\]

\item[(A4)] The total source term is locally integrable in time
\[
\int_0^\infty g_\omega(t)\, d\omega \in L^1_{\mathrm{loc}}[0,T].
\]
\end{itemize}

With these assumptions, we have the following result.

\begin{Thm}
Assuming that \emph{(A1)}--\emph{(A4)} hold, then equation \eqref{eq2} admits at least one solution.
\end{Thm}

The proof will be provided in the next section.\\

\textbf{Acknowledgments.} The author is thankful for Minh-Binh Tran's valuable guidance throughout this work. The author also thanks Jennifer Mackenzie for helping to edit this paper.
\end{section}

\begin{section}{Proof of Theorem 1}

We have the following equation
\begin{equation} \label{eq3}
\begin{split}
N' = 2N^2+g(t),
\end{split}
\end{equation}

obtained by integrating \eqref{eq2}  with respect to $\omega$, replacing $\int_0^\infty N_\omega d\omega$ with $N$, and let $g(t) =\int_0^\infty g_\omega(t) d\omega$. We note that~\eqref{eq3} is a Riccati equation with an arbitrary solution $N(t)$. Since $g(t)$ is non-negative, $N'(t)$ is non-negative. Additionally, given $N(t)$ has a non-negative initial value $N(0)$, we have the following remark.

\begin{Remark}
 The solutions to ~\eqref{eq3} are non-negative.
\end{Remark}

Similarly, by replacing $\int_0^\infty N_\mu d\mu$ with $N$ in equation \eqref{eq2} , we have

\begin{equation} \label{eq4}
\begin{split}
  \frac{\partial N_\omega}{\partial t}=  N_\omega N+ \int_0^\infty N_{\omega+\mu} N_\mu d\mu + g_\omega(t).\\
\end{split}
\end{equation}

Since $N$ is non-negative, \eqref{eq4} has a unique solution given by:\\
\begin{equation} \label{eq5}
    \begin{split}
        N_\omega= e^{\psi(t)}N_\omega(a) +e^{\psi(t)}\int_a^te^{-\psi(t)} \left( \int_0^\infty N_{\omega+\mu} N_\mu d\mu +g_w(t)\right) dt,
    \end{split}
\end{equation}

where $\psi(t)=\int_a^tNdt$. We aim to show that $\int_0^\infty N_\mu d\mu$ converges uniformly on $[0,T]$, but first we define the following auxiliary equations:
\[
       \begin{split}
           X(a):=\int_0^\infty N_\omega d\omega (a), \quad      g(t):=\int_0^\infty g_\omega(t)d\omega,
       \end{split}
\]

\[
\begin{split}
    d_\omega (t):= \sup_{a\leq s \leq t}N_\omega(s), \quad  m(t):= \sup_{a\leq s\leq t} e^{\psi(s)}\int_a^s e^{-\psi(\tau)} d\tau, \text{ and}
\end{split}
\]

\[
p_\omega(t) :=\sup_{a\leq s \leq t}N_\omega(a)e^{\psi(s)}+e^{\psi(s)}\int_a^s e^{-\psi(\tau)}g_\omega(\tau)d\tau.
\]

We prove the following lemma.

\begin{Lemma}
 Assuming conditions \emph{(A1)}--\emph{(A4)} hold, then for every $a \ge 0$, there exists $b>a$ such that
\[
f_n=\int_0^n N_\mu \, d\mu
\]
converges uniformly on the interval $[a,b]$.
\end{Lemma}
\begin{proof}

We consider $p_\omega(t)$. Since $\psi(s)$ is positive, non-decreasing, and both equations $N_\omega(a)$ and $g_\omega(\tau)$ are positive, we can replace the $e^{\psi(s)} \text{ by } e^{\psi(t)} $ and $  e^{-\psi(s)}$ by $1$. Hence,
\[
\begin{split} 
    p_\omega(t) \leq N_\omega(a)e^{\psi(t)}+e^{\psi(t)}\int_a^t g_\omega(\tau)d\tau.
\end{split}
\]

Hence, taking the integral on both sides with respect to $\omega$ yields\\
\[
\begin{split} 
   \int_0^\infty p_\omega(t) d\omega \leq & e^{\psi(t)} \int _0^\infty N_\omega d\omega (a)+e^{\psi(t)}\int_a^t \int_0^\infty g_\omega(\tau) d\omega d\tau\\
   =& X(a)e^{\psi(t)}+e^{\psi(t)}\int_a^t g(\tau) d\tau < \infty  .\\
\end{split}
\]

Note that the equation \eqref{eq5}  is bounded by\\
\[
\begin{split}
    p_\omega (t) +\sup_{a\leq s\leq t} e^{\psi(s)} \int_a^se^{-\psi(\tau)} \int_0^\infty N_{\mu+\omega}N_\mu d\mu d\tau.
\end{split}
\]

Taking the supremum over $t$ gives\\
\[
\begin{split}
   d_\omega(t) \leq p_\omega (t) + m(t) \int_0^\infty d_{\mu+\omega}d_\mu d\mu.
\end{split}
\]

Next, we define a family of functions $e_\omega(t)$ such that $e_0(t)=d_0(t)$ and\\
\[
\begin{split}
    e_\omega(t)  = p_\omega(t) + m(t) \int_0^\infty e_{\mu+\omega}(t)e_\mu (t)d\mu.
\end{split}
\]

We have the following generating functions\\
\[
\begin{split}
    E(t; z) & = \int_0^\infty e_\omega (t)  z^\omega  d\omega, \text{ and }\\
    P(t,z)&= \int_0^\infty p_\omega (t) z^\omega  d\omega  .\\
\end{split}
\]
Therefore,
\[
\begin{split}
E(t,z)& = P(t;z) + m E(t,z)^2.
\end{split}
\]

Thus,\\
\[
\begin{split}
    E (t,z) = \frac{1\pm\sqrt{1-4Pm}}{2m}.
\end{split}
\]

We note that $m(t) \rightarrow 0$ as $t \rightarrow a^+$. Furthermore, since $\int_0^\infty p_\omega(t)d\omega$ is finite, there is $b$ close to $a$ such that $\int_0^\infty p_\omega(t)d\omega \leq \frac{1}{4m(t)}$ for all $t\in [a, b]$ . Then, if $|z| \leq 1$, we have:\\
$$ |4m(t)P(t,z)| \leq 4m(t) \int_0^\infty p_\omega(t) (|z|^\omega)d\omega \leq  4m(t) \int_0^\infty p_\omega(t)d\omega \leq 1.$$\\

Hence, $E(t,1)$ admits real solutions for all $t\in [a, b]$. That is:
\[
\int_0^\infty e_\omega(t) d\omega<\infty, \text{ for all }t\in[a,b].
\]
We also have 
$$\int_0^ \tau N_\omega d\omega \leq \int_0^\tau d_\omega d\omega \leq \int_0^\infty e_\omega d\omega  \leq  \infty, \text{ for all $\tau\in \mathbb{R}$}. $$\\
Let $I_i = [i-1, i) \text{ for } i\in \mathbb{N}$. We define a sequence of function $\{\int_{I_i}N_\omega\}_i $ and $\{\int_{I_i}p_\omega\}_i$. Hence, for each $i\in I$, we have\\
$$ \int_{I_i}N_\omega\ d\omega \leq \int_{I_i}p_\omega d\omega\ \leq \int_{I_i} e_\omega d\omega (b).$$\\
Therefore, by the Weierstrass M-test, $\int_{0}^\infty N_\omega d\omega$ converges uniformly.
\end{proof}

   From this lemma, we know that $\int _0^\infty N_\omega d\omega$ converges uniformly in $[0,b]$ for some $b > 0$. We let $M(t) = \int _0^\infty N_\omega d \omega$. We obtain the following result.\\

\begin{Lemma}
   Assuming conditions \emph{(A1)}--\emph{(A4)} hold. Then there exists $b>0$ such that
\[
M(t)=N(t), \quad \text{for all } t\in[0,b].
\]
\end{Lemma}
\begin{proof}
    Let $M_n(t) = \int_0^n N_\omega d\omega$ and $G_n(t) = \int_0^n g_\omega d\omega$.\\
    
    By integrating \eqref{eq5} with respect to $\omega$, we have\\
    $$M_n(t) =e^{\psi(t)}\int_0^nN_\omega (0) d\omega + e^{\psi(t)}\int_0^t e^{-\psi(t)}\int_0^n\int_0^\infty N_{\omega+\mu}N_\mu d\mu d\omega dt +e^{\psi(t)} \int_0^t e^{-\psi(t)}\int_0^n g_\omega(t) d\omega dt.$$\\

    Now, as $\int_0^nN_\omega d\omega$ is continuous with respect to $t$ and converges uniformly to $M(t)= \int_0^\infty N_\omega d\omega$, the Uniform Limit Theorem tells us that $M(t)$ is continuous with respect to $t$. In addition, we have\\
    \[
    \begin{split}
         \int_0^\infty\int_0^n  N_{\omega+\mu}d\mu N_\omega d\omega &= \int_0^\infty\int_\omega^{n+\omega}  N_t dt N_\omega d\omega\\
         & \leq  M^2(t).
    \end{split}
    \]

On the other hand, $$e^{-\psi(t)}\int_0^ng_\omega(t) d\omega \le \int_0^ng_\omega(t) d\omega \leq g(t).$$
Therefore, by the Dominated Convergence Theorem,\\
$$M(t) =e^{\psi(t)} \left( N (0) + \int_0^t e^{-\psi(\tau)}\int_0^\infty\int_0^\infty N_{\omega+\mu}N_\mu d\mu d\omega dt +\int_0^t e^{-\psi(\tau)} g(\tau)  dt\right).$$\\
By taking the derivative of $M(t)$ with respect to $t$, we have\\
\[
\begin{split}
    M'(t)&=N(t)M(t)+M^2+g(t).
\end{split}
\]

Let $g(t)=N'-2N^2$ as in \eqref{eq3}. Then

\[
M'(t)=N(t)M(t)+ M^2(t)+N'-2N^2.
\]

Hence,
\[
(M-N)'(t) = (M-N)(M+2N).
\]

With the initial condition $M(0)=N(0)$, we conclude that $M(t)=N(t)$ on the interval $[0,b].$
\end{proof}

In the next Lemma, we extend the result proved in \textbf{Lemma 4} to $[0,T]$.
\begin{Lemma}
     If conditions \emph{(A1)}--\emph{(A4)} hold, then $M(t)= N(t)$ on $[0, T]$.
\end{Lemma}

\begin{proof}
    Let $B=\{\beta \in [0,T] \ | \ M(t)=N(t) \text{ for all } t\in [0, \beta] \}$, and let $b=\sup B$. We note that $b\not \in B$. Since if $b\in (0,T]$, by \textbf{Lemma 4}, there is some $c >b$ such that $M(t)=N(t) \text{ on } [b,c]$. This implies that $M(t)=N(t) \text{ on } [0,c]$ which contradicts the definition of $b$.\\

     There are two cases, $b<T \text{ or } b=T$.\\
   
    If $b<T$, we have $\{M_n(t)\}_{n=1}^\infty$ is an increasing sequence satisfying:
   \[
   M_n(t)\leq M(t)= N(t), \ \text{ for all } t\in [0,b)
   .\]
Thus, because $M_n(t)$ and $N(t)$ is continuous on $[0,T]$, we obtain
\[
M_n(b)\leq N(b).
\]
  By taking the limit with respect to $n$ on both sides, we arrive at $M(b)= N(b)$, which implies $b\in B$. This is a contradiction to what was shown earlier.
   
   Therefore, $b=T$.\\
\end{proof}

\textbf{Proof of Theorem 1.} It is clear that \eqref{eq4}  is the solution for \eqref{eq1}, since $\int_0^\infty N_\omega d\omega$ converges to $N(t)$ on $[0,T]$. We have established a solution for the third term of the Connaugton-Newell Model with a source term. \hfill\qedsymbol
\end{section}
\def\cprime{$'$} \def\cprime{$'$}
\providecommand{\bysame}{\leavevmode\hbox to3em{\hrulefill}\thinspace}
\providecommand{\MR}{\relax\ifhmode\unskip\space\fi MR }
\providecommand{\MRhref}[2]{%
  \href{http://www.ams.org/mathscinet-getitem?mr=#1}{#2}
}
\providecommand{\href}[2]{#2}

\end{document}